\documentclass[11pt,letterpaper]{amsart}
\usepackage{amscd}
\usepackage{amssymb}
\usepackage{amsthm}
\usepackage{enumerate}
\usepackage[normalem]{ulem}
\usepackage{mathtools}
\usepackage[usenames,dvipsnames]{xcolor}
\usepackage{stmaryrd}
\usepackage[left=1in,top=1in,right=1in]{geometry}
\usepackage[mathscr]{eucal}
 \usepackage{cite}
\usepackage{upgreek}
\usepackage[bookmarks=false]{hyperref}
\usepackage[T2A]{fontenc}
\usepackage[utf8]{inputenc}
\usepackage{mathrsfs}

\newtheorem{introthm}{Theorem}
\newtheorem{introcor}[introthm]{Corollary}

\newtheorem{thm}{Theorem}[section]
\newtheorem{lem}[thm]{Lemma}
\newtheorem{prop}[thm]{Proposition}

\theoremstyle{definition}
\newtheorem{defn}[thm]{Definition}

\theoremstyle{remark}

\numberwithin{equation}{section}

\newcommand{\bN}{{\mathbb N}}

\newcommand{\bR}{{\mathbb R}}

\newcommand{\bZ}{{\mathbb Z}}

\newcommand{\cC}{{\mathcal C}}

\newcommand{\cH}{{\mathcal H}}

\newcommand{\cU}{{\mathcal U}}

\DeclareMathOperator{\Span}{span}

\newcommand{\ip}[1]{\langle #1 \rangle}

\newcommand{\ee}{\varepsilon}

\begin{document}

\title{Simplices of maximally amenable extensions in II$_1$ factors}

\author{Srivatsav Kunnawalkam Elayavalli}
\address{\parbox{\linewidth}{Department of Mathematics, University of California, San Diego, \\
9500 Gilman Drive \# 0112, La Jolla, CA 92093}}
\email{srivatsav.kunnawalkam.elayavalli@vanderbilt.edu}
\urladdr{https://sites.google.com/view/srivatsavke}

\author{Gregory Patchell}
\address{\parbox{\linewidth}{Department of Mathematics, University of California, San Diego, \\
9500 Gilman Drive \# 0112, La Jolla, CA 92093}}
\email{gpatchel@ucsd.edu}
\urladdr{https://sites.google.com/view/gpatchel}

\begin{abstract}
For every $n\in \mathbb{N}$  we obtain a separable II$_1$ factor $M$ and a maximally abelian subalgebra $A\subset M$ such that the space of maximally amenable extensions of $A$ in $M$ is affinely identified with the $n$ dimensional $\mathbb{R}$-simplex.  This moreover yields first examples of masas in II$_1$ factors $A\subset M$ admitting exactly $n$ maximally amenable factorial extensions. Our examples of such $M$ are group von Neumann algebras of free products of lamplighter groups amalgamated over the acting group. A conceptual ingredient that goes into obtaining this result is a \emph{simultaneous relative} asymptotic orthogonality property, extending prior works in the literature. The proof uses technical tools including our \emph{uniform-flattening} strategy for commutants in ultrapowers of II$_1$ factors.
\end{abstract}

\maketitle%


\section{Introduction}

The study of maximally amenable subalgebras has provided inroads to uncovering the structure of II$_1$ factors and has thus seen consistent interest from experts since the early days of the subject. While the existence of such maximally amenable subalgebras follows from Zorn's lemma, it is very unclear as to their structure and position inside the ambient factor. A key early question of Kadison \footnote{Problem 7 from `Problems on von Neumann algebras,' Baton Rouge Conference } asked if every maximally abelian subalgebra (masa) in a II$_1$ factor has to be contained in a copy of the hyperfinite II$_1$ factor. This was settled in the negative by Popa in \cite{Popa1983MaximalIS} where it was shown that the generator masa in the free group factors is in fact maximally amenable. Since this result, there has been a variety of other results proving maximal amenability and absorption of certain canonical subalgebras \cite{WenAOP,2AuthorsOneCup, CFRW, HOUDAYER2014414, Bleary} building on the idea of asymptotic orthogonality from \cite{Popa1983MaximalIS}. Other new ideas for studying maximally amenable subalgebras have since been developed: using studies of central states \cite{BC2015, OzAbsor}, conditional expectations \cite{BH18}, 1-bounded entropy theory \cite{freePinsker}, and upgraded asymptotic freeness \cite{jekel2024upgradedfreeindependencephenomena}.  A crowning conjecture about the structure of maximally amenable subalgebras in free group factors was stated by Peterson and Thom \cite{PetersonThom}. This states that every diffuse amenable subalgebra of $L(\mathbb{F}_2)$ is contained in a \emph{unique} maximally amenable subalgebra. Due to breakthrough works of Hayes \cite{HayesPT} using his 1-bounded entropy techniques \cite{Hayes2018}, and a recent flurry of strong convergence results \cite{bordenave2023norm, PTkilled, magee2024strongasymptoticfreenesshaar, parraud2024spectrumtensorrandomdeterministic, chen2024newapproachstrongconvergence}, this conjecture has been resolved in the positive (see also \cite{HJKEPT} where striking further consequences are discussed). Note that recently, an analogue of this conjecture involving maximally amenable extensions has also been proved in the Type III setting (see \cite{hayes2024generalsolidityphenomenaanticoarse}). 

Motivated by the above phenomena of uniqueness of maximally amenable extensions, it is very natural to wonder about a more quantitative situation. Namely, is it possible to completely understand the space of maximally amenable extensions of masas, and does this study permit any interesting finiteness behaviors. Despite the breadth of the literature surrounding these problems, this particular thread has not been investigated.  In this paper, we make a contribution to this question by fully studying the structure of amenable extensions of certain masas inside certain amalgamated free product II$_1$ factors, opening up seemingly exotic scenarios: 

\begin{introthm}\label{beast1}
Let $N_i \simeq L(\mathbb Z \wr \mathbb Z)$ for $1\leq i\leq k.$ Form the amalgamated free product $M = *_A N_i$ where $A\subset N_i$ is the von Neumann algebra $L(\mathbb Z)$ of the acting group. Then any amenable extension $P\supset A$ contains central projections $p_1,\ldots,p_k \in Z(P)\subset A$ such that $\sum_i p_i = 1$ and $Pp_i \subset p_iN_ip_i$ for all $1\leq i\leq k$. In particular, $A$ has exactly $k$ maximal amenable extensions in $M$ which are factors.
\end{introthm}

As a consequence of the above result, we are able to identify a natural convex structure on the \emph{space} of maximally amenable extensions of $A\subset M$. To make this precise, we would have to quotient out by unitary conjugacy. We document this below: 

\begin{introcor}\label{coro of beast}
    Let $N_i \simeq L(\mathbb Z \wr \mathbb Z)$ for $1\leq i\leq k.$ and $M = *_A N_i$ where $A\subset N_i$ is the von Neumann algebra $L(\mathbb Z)$ of the acting group. Then the set of maximally amenable subalgebras containing $A$ modulo unitary conjugation in $M$ is affinely identified with the $k$ dimensional $\mathbb{R}$-simplex. 
\end{introcor}

We remark that this result also holds if we do a countably infinite amalgamated free product, in which case the associated simplex is the one with countably infinitely many separated extreme points. Before we move on to describe the key ideas in the proof of our main result, we describe some structural properties of such factors $M$, and an interesting question that is opened up by our work. Firstly they are not free group factors since they are strongly 1-bounded by the join lemma \cite{JungSB}, in fact have $h(M)=0$ in the sense of Hayes \cite{Hayes2018} (note $M$ is Connes embeddable by \cite{ESZ2}). However by results of Ioana \cite{CartanAFP} they are strongly solid (see also \cite{Cyrilsneaky}) and therefore non Gamma (see Proposition \ref{hopefully well known to experts}). Our work inspires a key open question, in the spirit of the free group factor problem, of whether the factors $*_A (N_i)_{i=1}^k$ are isomorphic for different values of k.. Say $m(A)$ is the number of factorial maximal amenable extensions of $A$ in $M$. Our result yields a potential invariant for II$_1$ factors, namely $\theta(M)= \sup_{masa\ A\subset M} m(A)$, and we suspect that this invariant should yield exactly the values $k$ for our factors $M$, thereby distinguishing them. However we are unable to settle this question at present.

\subsection*{Comments on proofs} First we describe the broad strokes for Theorem \ref{beast1}. Our work is inspired in equal parts by the paper of Popa \cite{Popa1983MaximalIS} and by the paper of Houdayer \cite{HOUDAYER2014414}. However, there are new ideas at play in our result, including a crucial use of mixingness in the form of malnormality and the precise structure of the lamplighter groups, and certain ultrapower techniques that were previously developed by us in \cite{patchellelayavalli2023sequential}. The proof splits into two parts: a technical part and a soft part. The technical part involves proving a generalization of asymptotic orthogonality that simultaneously takes into account the span of each of the lamplighter group von Neumann algebra in our amalgamated free product (see \ref{aopsimult}). In order to prove this technical phenomena in our setting, we crucially use our \emph{uniform-flattening} strategy from \cite{patchellelayavalli2023sequential}, where we precisely identify the structure of lifts of commutants in ultrapower II$_1$ factor via careful use of Kaplansky's density theorem. Despite being technical in nature, the conceptual pieces of this argument are isolated to the best of our ability in the third paragraph of proof of Theorem \ref{thm: asymptotic orthogonality}. The second part of the argument, which is the soft part, is an adaptation of work of Houdayer in \cite{HOUDAYER2014414}. We show that in the presence of mild mixing conditions, our simultaneous relative asymptotic orthogonality property yields the structural classification of maximally amenable extensions of $A$. The proof of this result uses an amplification trick pertaining to intertwining. Finally Corollary \ref{coro of beast} is obtained using elementary arguments: patching and Murray-von Neumann equivalence of projections, and also an easy non intertwining calculation (see for instance \cite{gao2024conjugacyperturbationsubalgebras}). We would like to point out that we suspect that Theorem \ref{beast1} may hold in the generality of just considering $N_i\cong R$ and $A$ being an arbitrary mixing masa in $N_i$. As we outline above, in our main result we crucially use the structure of the lamplighter groups to effectively navigate the technical underpinnings of the proof. In the general setting, we do not immediately see an elegant way to tame the technicalities, hence we do not pursue it here.


\subsection*{Acknowledgements} It is our pleasure to thank Matt Kennedy for sharing very insightful suggestions, and Changying Ding for some crucial early discussions and helpful suggestions. We also thank David Jekel for his extremely helpful comments and encouragement. Parts of this work was completed when the first author visited University of Virginia in September 2024, and when the second author visited Banff International Research Station for the Group Operator Algebras program in September 2024. We thank these institutions for their hospitality. The first author acknowledges support from the NSF grant DMS 2350049. The second author was supported in part by NSF grant DMS 2153805.

\section{Preliminaries} 

Let $(M,\tau)$ be a tracial von Neumann algebra, i.e., a pair consisting of a von Neumann algebra $M$ and a faithful normal tracial state $\tau:M\rightarrow\mathbb C$. We denote by $\mathcal U(M)$ the group of unitaries of $M$ and by $M_{\text{sa}}$ the set of self-adjoint elements of $M$. For an ultrafilter $\mathcal U$ on a set $I$, we denote by $M^{\mathcal U}$ the tracial ultraproduct: the quotient $\ell^\infty(I,M)/\mathcal{J}$ by the closed ideal $\mathcal{J}\subset\ell^\infty(I,M)$  consisting of $x=(x_n)$ with $\lim\limits_{n\rightarrow\mathcal U}\|x_n\|_2= 0$.  
We have a natural diagonal inclusion $M\subset M^{\mathcal U}$ given by $x\mapsto (x_n)$, where $x_n=x$, for all $n\in I$. We recall the amalgamated free product construction for tracial von Neumann algebras.
Let $(M_1,\tau_1)$ and $(M_2,\tau_2)$ be tracial von Neumann algebras with a common von Neumann subalgebra $B$ such that ${\tau_1}_{|B}={\tau_2}_{|B}$. We denote by $M=M_1*_{B}M_2$ the amalgamated free product with its canonical trace $\tau$. See \cite{Popa93} and \cite{VDN1992} for more details on the construction.

Let $M$ be a finite von Neumann algebra and $N\subset M$ a von Neumann subalgebra. Recall the inclusion $N\subset M$ is mixing if $L^2(M\ominus N)$ is mixing as an $N$-$N$ bimodule, i.e., for any sequence $u_n\in\mathcal U(N)$ converging to $0$ weakly, one has $\|E_N(xu_ny)\|_2\to 0$ for any $x,y\in M\ominus N$.
When $M$ and $N$ are both diffuse, we may replace sequence of unitaries with any sequence in $N$ converging to $0$ weakly \cite[Theorem 5.9]{DKEP21}. Examples of mixing subalgebras include $M_1$ and $M_2\subset M_1\ast M_2$, where $M_1$ and $M_2$ are diffuse \cite[Proposition 1.6]{jolissaint}  and $L\Lambda\subset L\Gamma$, where $\Lambda<\Gamma$ is almost malnormal (see Proposition 2.4 in \cite{boutonnetcarderimaxamen}). The following is elementary: 

\begin{lem}\label{lem: amalgam-mixing}
    If $A\subset M_i$, $i=1,2$ are mixing inclusions of a von Neumann subalgebra $A$ into finite von Neumann algebras $M_i$, then the inclusion of $A$ into the amalgamated free product $A\subset M_1*_AM_2$ is mixing.
\end{lem}

\begin{proof}
    Let $(u_n)$ be a sequence of unitaries in $A$ converging to 0 weakly. It suffices to check that $\|E_A(xu_ny)\|_2\to 0$ on a $\|\cdot\|_2$-densely spanning set of norm one elements in $M_1*_AM_2 \ominus A.$ $M_1*_AM_2\ominus A$ is densely spanned by words $t_1\cdots t_n$ where $t_j \in M_{i_j}\ominus A$ and $i_{j}\neq i_{j+1}$ for all $j.$ That is, the $t_j$ alternatively come from $M_1\ominus A$ and $M_2\ominus A$. Write $x = x_1\cdots x_r$ and $y = y_1\cdots y_s$ where the $x_j$ and $y_j$ alternatively come from $M_1\ominus A$ and $M_2\ominus A$. There are two cases to consider.

    First, if $x_r \in M_1\ominus A$ and $y_1\in M_2\ominus A$ (or vice versa), then for any $u\in A,$ $x_ru \in M_1\ominus A$ and so $x_1\cdots x_{r-1}(x_ru)y_1\cdots y_s$ is a reduced word in $M_1*_AM_2$ and thus is orthogonal to $A$. In other words, in this case $E_A(xu_ny) = 0$. Second, suppose $x_r,y_1 \in M_1\ominus A$ (or that they are both in $M_2\ominus A$). Since $A\subset M_1$ is mixing, $\|E_A(x_ru_ny_1)\|_2 \to 0.$ Write $x_ru_ny_1 = E_A(x_ru_ny_1) + z_n$ so that $z_n \in M_1\ominus A$. Then as $n\to\infty,$
    \begin{align*}
        \|E_A(xu_ny)\|_2^2 &= \|E_A(x_1\cdots x_ru_ny_1\cdots y_s)\|_2^2 \\
        &= \|E_A(x_1\cdots x_{r-1}E_A(x_ru_ny_1)y_2\cdots y_s)\|_2^2 + \|E_A(x_1\cdots x_{r-1}z_ny_2\cdots y_s)\|_2^2\\
        &\leq \|x_1\cdots x_{r-1}E_A(x_ru_ny_1)y_2\cdots y_s\|_2^2\\
        &\leq \|E_A(x_ru_ny_1)\|_2^2 \|x_1\cdots x_{r-1}\|^2\|y_2\cdots y_s\|^2 \to 0,
    \end{align*}
    demonstrating that the inclusion $A\subset M_1*_AM_2$ is mixing.
\end{proof}

We recall the following additional notion appearing in \cite{popa2008strong}: Let $M$ be a finite von Neumann algebra and $A\subset N\subset M$ be von Neumann subalgebras. We say that $N\subset M$ is weakly mixing through $A$ if there exists a net of unitaries $u_k\in \mathcal{U}(A)$ such that $$\lim_k \|E_N(xu_ky)\|_2\to 0\ \forall x,y\in M\ominus N. $$ 

The following is the well known intertwining theorem from \cite{PopaStrongRigidity}, which we will need in our proof. 

\begin{thm}[\cite{PopaStrongRigidity}]
\label{thm-popa-fundamental}
Let $(M,\tau)$ be a tracial von Neumann algebra, $f\in M$ a nonzero projection, and $P,Q$ be von Neumann subalgebras of $fMf$ and $M$, respectively. Then the following are equivalent:

\begin{enumerate}
    \item[(a)] There is no net $(u_i)$ of unitary elements in $P$ such that for every $x,y\in M,$ $\lim_i \|E_Q(x^*u_iy)\|_2=0$;
    \item[(b)] There exists a nonzero $P$-$Q$-subbimodule $\cH$ of $fL^2(M)$ such that $\dim(\cH_Q)<\infty$;
    \item[(c)] There exists an integer $n\geq 1$, a projection $q\in M_n(Q)$, a nonzero partial isometry $v\in M_{1,n}(fM)$ and a normal unital homomorphism $\theta\colon P\to qM_n(Q)q$ such that $v^*v\leq q$ and $xv = v\theta(x)$ for all $x\in P$;
    \item[(d)] There exist nonzero projections $p\in P$ and $q\in Q$, a unital normal homomorphism $\theta\colon pPp\to qQq$, and a nonzero partial isometry $v\in pMq$ such that $xv=v\theta(x)$ for all $x\in pPp.$ Moreover, $vv^* \in (pPp)'\cap pMp$ and $v^*v\in \theta(pPp)'\cap qMq.$
\end{enumerate}
\end{thm}

\begin{lem}\label{lem: wk-mixing-absorb}
    If $N\subset M$ is weakly mixing through $A$ and $x\in M$ is such that there exist $y_1,\ldots,y_n$ satisfying $Ax \subset \sum_{j=1}^n y_jN$, then $x\in N.$ In particular, $A'\cap M\subset N.$
\end{lem}

\begin{proof}
    This is an immediate application of Proposition 6.14 of \cite{popa2008strong}.
\end{proof}

\begin{lem}\label{prop: masa}
    Let $N_i \simeq L(\mathbb Z \wr \mathbb Z)$ for $1\leq i\leq k.$ Form the amalgamated free product $M = *_A N_i$ where $A\subset N_i$ is the von Neumann algebra $L(\mathbb Z)$ of the acting group. Then $A$ is a MASA in $M.$ 
\end{lem}

\begin{proof}
    This follows immediately from the above Lemma \ref{lem: wk-mixing-absorb} and Lemma \ref{lem: amalgam-mixing}.
\end{proof}

\begin{thm}\label{thm: adrian}
    Let $N_i \simeq L(\mathbb Z \wr \mathbb Z)$ for $1\leq i\leq k.$ Form the amalgamated free product $M = *_A N_i$ where $A\subset N_i$ is the von Neumann algebra $L(\mathbb Z)$ of the acting group. Then $M$ is a non Gamma II$_1$ factor; it is moreover strongly solid.
\end{thm}

\begin{proof}
    This is immediate from Theorem 1.8 in \cite{CartanAFP}, see also \cite{Cyrilsneaky}.
\end{proof}

\begin{thm}[\cite{Bleary}]\label{thm: BLeary}
    Let $N_i \simeq L(\mathbb Z \wr \mathbb Z)$ for $1\leq i\leq k.$ Form the amalgamated free product $M = *_A N_i$ where $A\subset N_i$ is the von Neumann algebra $L(\mathbb Z)$ of the acting group. Then each $N_i$ is maximal amenable in $M$.
\end{thm}

\begin{proof}
    This is nearly immediate from Theorem 1.1 of \cite{Bleary}; note that $N_i \not\prec_{N_i} A$ for all $i.$ See also \cite{BH18}. 
\end{proof}

The following is well known to experts. We thank Changying Ding for pointing out to us that it is an immediate consequence of Theorem 3.1 in \cite{Uedafree}. The argument we provide here comes from salvaging the idea in Proposition 7 of \cite{OzawaSolidActa}, which unfortunately remains incorrect as stated. We nevertheless include this proposition here for the benefit of the community.

\begin{prop}\label{hopefully well known to experts}
    If $M$ is a solid and nonamenable II$_1$ factor, then $M$ is non-Gamma.
\end{prop}

\begin{proof}
    Suppose that $M$ has property Gamma and is nonamenable. We will show that $M$ is not solid. There is a finite set $\{x_1,\ldots,x_n\}$ in $M$ such that $\{x_1,\ldots,x_p\}''$ is a nonamenable subalgebra of $M$. By \cite{Connes} there is $\ee>0$ such that $\|x_i'-x_i\|_2 < \ee$ for all $i$ implies $\{x_1',\ldots,x_p'\}''$ is also nonamenable. We claim that there exist self-adjoint trace 0 unitaries $u_{m,n}\in M$ (where $m\le n$) with the following properties:

    \begin{enumerate}
        \item For all $\ell,m\le n,$ $[u_{\ell,n},u_{m,n}] = 0$;
        \item For all $\ell\neq m\le n$, $\tau(u_{\ell,n}u_{m,n})=0$;
        \item For all $m\le n,$ $\|u_{m,n} - u_{m,n+1}\|_2 < 2^{-n}$;
        \item For all $i\leq p$ and $m\leq n$, we have $\|[x_i,u_{m,n}]\|_2 < \ee(2^{-m} - 2^{-(n+1)})$.
    \end{enumerate}

    To prove the claim we proceed by induction on $n$. The case $n=1$ is clear since $M$ has property Gamma (here, we only need to check condition (4), and we use the fact that we can lift trace 0 self-adjoint unitaries to trace 0 self-adjoint unitaries). So assume unitaries $u_{k,\ell}$ satisfying (1)-(4) exist for all $k\leq\ell\leq n.$ Since $M'\cap M^\cU$ is diffuse, we can find a self-adjoint trace 0 unitary $v\in M'\cap M^\cU$ which is orthogonal to all of the $u_{m,n}$ with $1\leq m\leq n.$ Now the $u_{m,n}$ and $v$ are in the same abelian subalgebra of $M^\cU$ so that we can lift them to sequences of commuting, orthogonal, trace 0 self-adjoint unitaries: $u_{m,n} = (a_{m,k})_k$ and $v = (b_k)_k$ (see Lemma 2.2 of \cite{gao2024internal}). Choose $K$ sufficiently large so that \begin{itemize}
        \item For all $j\leq p$ and $m\leq n$, $\|[x_j,a_{m,K}]\|_2 < \ee(2^{-m} - 2^{-(n+2)})$;
        \item For all $j\leq p$, $\|[x_j,b_{K}]\|_2 < \ee(2^{-(n+1)} - 2^{-(n+2)})$;
        \item For all $m\leq n,$ $\|u_{m,n} - a_{m,K}\|_2 < 2^{-(n+1)}$.
    \end{itemize}
    This is possible due to the definition of ultrapowers and since we only have to verify finitely many approximate conditions. Now set $u_{m,n+1} = a_{m,K}$ for $m\leq n$ and set $u_{n+1,n+1} = b_K$. It is immediate these satisfy (1)-(4) along with the previously created $u_{k,\ell}$. This proves the claim.

    Now take, for each $m\in\bN,$ $u_m = \lim_n u_{m,n}$. This limit exists in 2-norm because we constructed Cauchy sequences. Furthermore, we have that each $u_m$ is a trace 0 self-adjoint unitary in $M$, all of the $u_m$ commute, and all of the $u_m$ are orthogonal. Now, denote $A_m = \{u_m\}''$ and $B_m = \{u_1,\ldots,u_m\}''$. Since $\|x_j-u_mx_ju_m\|_2 \leq \ee2^{-m}$ for all $1\leq j\leq p$ and $m\geq1,$ we see that $\|\frac{1}{2}(x_j+u_mx_ju_m) - x_j\|_2 \leq \ee2^{-(m+1)}$. Therefore $$\|E_{A_m'\cap M}(x_j) - x_j\|_2 \leq \ee 2^{-(m+1)}$$ for all $m\geq 1$ and $1\leq j\leq p.$    We note that $E_{B_m'\cap M}\circ E_{A_{m+1}'\cap M} = E_{B_{m+1}'\cap M}$ for all $m\geq 1$ and that furthermore, denoting $B = \overline{\cup_mB_m}$, for each $y\in M$, $\|E_{B'\cap M}(y) - E_{B_m'\cap M}(y)\|_2 \to 0$ (see Lemma 1.2 of \cite{Pop81kadison}). Therefore we get that, for each $1\leq j\leq p,$ 
    \begin{align*}
        \|E_{B_{m+1}'\cap M}(x_j) - E_{B_m'\cap M}(x_j)\|_2 &= \|E_{B_m'\cap M}(E_{A_{m+1}'\cap M}(x_j)) - E_{B_m'\cap M}(x_j)\|_2  \\
        &\leq \|E_{A_m'\cap M}(x_j) - x_j\|_2 \\
        &\leq \ee 2^{-(m+2)}.
    \end{align*} 
    A repeated use of the triangle inequality shows that $\|E_{B'\cap M}(x_j) - x_j\|_2 \leq \ee$ for each $1\leq j\leq p$. For $j\leq p,$ set $x_j' = E_{B'\cap M}(x_i)$. It follows that $\|x_j'-x_j\|_2 \leq \ee$ for all $j\leq p.$ Therefore $\{x_1,\ldots,x_p\}''$ is a nonamenable subalgebra of $M$ with diffuse commutant.
\end{proof}

\section{Simultaneous relative asymptotic orthogonality}

Recall the following definition of relative asymptotic orthogonality property from \cite{HOUDAYER2014414}, inspired by \cite{Popa1983MaximalIS}, see also \cite{Bleary}.

\begin{defn}\label{aopsimult}
    For von Neumann algebras $P\subset N\subset M$, we say that the inclusion $N\subset M$ has the asymptotic orthogonality property relative to $P$ if for any non principal ultrafilter $\mathcal{U}$ on $\mathbb{N}$, for any $y_1,y_2 \in M \ominus N$, and for any $x\in P' \cap(M^{\mathcal{U}}\ominus N^\cU)$, we have that $xy_1 \perp y_2x$ in $L^2(M^{\mathcal{U}})$. 
\end{defn}

For the purpose of our results we need to investigate the following setting of simultaneous relative asymptotic orthogonality for a family of subalgebras:

\begin{defn}
    Let $(M,\tau)$ be a tracial von Neumann algebra and let $\cU$ be a free ultrafilter on $\mathbb N.$ Let $A\subset N_i \subset M$ be von Neumann subalgebras where $i$ ranges over some index set $I$. We say that $\{N_i\}_{i\in I}$ has the asymptotic orthogonality property in $M$ relative to $A$ if for all $x\in A' \cap M^\cU\ominus (\overline{\Span}\{N_i^\cU : i\in I\})$ and $y_1,y_2\in M\ominus (\overline{\Span}\{N_i : i\in I\})$ we have $xy_1 \perp y_2x$ in $L^2(M^\cU)$.
\end{defn}

The main result of this section is the following:
\begin{thm}\label{thm: asymptotic orthogonality}
    Let $N_i \simeq L(\mathbb Z \wr \mathbb Z)$ for $i\in I,$ where $I$ is some index set. Form the amalgamated free product $M = *_A N_i$ where $A\subset N_i$ is the von Neumann algebra $L(\mathbb Z)$ of the acting group. Then $\{N_i\}_{i\in I}$ has the asymptotic orthogonality property in $M$ relative to $A$.
\end{thm}

\begin{proof} 



    We prove the case $|I|=2$ with the other cases being similar. Let $G = \bZ\wr\bZ *_\bZ \bZ\wr \bZ$ where the amalgam is the acting group of both wreath products. Denote by $K\simeq \bZ$ the acting group and $H_1,H_2$ the two wreath products containing $K$. Denote by $a$ the generator of $K$ and $b_n,c_n$ the commuting generators of the spatial components of $H_1,H_2$ respectively. An element in $G$ can be uniquely represented by a word of the form $B_1C_1\cdots B_nC_na^p$ where $B_i \in H_1$ is a product of nonzero powers of the $b_m$ where the $b_m$ are listed in increasing order of $m$ and similarly for each $C_i$. We call each $B_i$ and $C_i$ a \emph{block}. The \emph{coordinates} of a block $B$ (resp. $C$) are the integers $m$ such that some nonzero power of $b_m$ (resp. $c_m$) appears as a factor of $B$ (resp. $C$). We also note that we have $A = L(K)$, $N_i = L(H_i),$ and $M = L(G),$ and we identify the group elements in $G$ with the corresponding group unitaries in $M$.

    Let $y_1,y_2 \in G\setminus (H_1\cup H_2)$. We assume that all coordinates in all blocks of $y_1$ and $y_2$, as well as the powers of $a$ at the end of the words $y_1$ and $y_2$, are in the interval $[-N,N]$. Let $x \in A' \cap M^\cU \ominus (N_1^\cU + N_2^\cU)$ and assume $\|x\|_2=1$. Our present aim is to show that $y_1x\perp xy_2$. To that end, we will find an approximation of a lift $(x_k)_k$ of $x$ such that $y_1x_k \perp x_ky_2$ for all $k.$

    Before diving into the technicalities, we outline our strategy. First, we use the \emph{uniform-flattening} strategy from \cite{patchellelayavalli2023sequential} to say that $x$ lifts to elements whose Fourier coefficients are extremely close to being invariant under conjugation by $a$. This forces two things. First, the first blocks in words in the support of $x$ nearly always have some coordinate outside of a fixed very large interval $I$ ($I$ depending on $N$ and $\ee$). Second, the first blocks in words in the support of $x$ cannot cancel with $y_1$ without leaving behind nonzero coordinates in a large interval $J\subset I$. This means that the words in the support of the product $y_1x$ start with blocks whose coordinates are in $[-N,N]$, their second blocks contain some coordinate in $J,$ and have at least 3 blocks. On the other hand, consider the situation for $xy_2$. A word in its support can only start with a block with coordinates in $[-N,N]$ if all but one of the blocks coming from $x$ get fully cancelled by a shift of $y_2$. Since $x$ has a starting block with a very large coordinate (outside $I$), this means that the shift is very large. Hence the second block either doesn't exist or comes from a very large shift of $y_2,$ meaning its coordinates will lie solely outside of $J$. This forces $y_1x$ and $xy_2$ to be orthogonal. Taking sums and limits proves the result for $y_1,y_2$ not words.

Fix $\ee>0.$ Let $\cC\subset G\setminus K$ be a set such that every element of $G\setminus K$ can be obtained via conjugating by some power of $a$ but no two elements in $\cC$ are conjugate via a power of $a$. Furthermore, we may assume that whenever $g\in G\setminus K$ is a word such that its first block contains coordinates only in $[-N,N],$ then $g = a^nwa^{-n}$ for some $w\in \cC$ and some $n\in[-N,N]$. Let $J = [-S,S]$ and $I = [-R,R]$ where $R,S\in\bN$ are such that $S > N + (2N+1)/\ee^2$ (so that $(S-N)/(2N+1) > 1/\ee^2$) and $R = S+2N$. For each element $g\in G\setminus K,$ define $c(g)$ to be the set of coordinates of the first block of $g$. Define 
$$X_g = \{n\in\bZ : \emptyset \neq c(a^nga^{-n})\cap[-S,S] \subset [-N,N]\}.$$
Note that if $n\in X_g$ and $k\in [2N+1, S-N]$ then $x+k \not\in X_g$. Furthermore, if $n_1,n_2\in X_g$ then either $|n_1-n_2| \leq 2N$ or $|n_1-n_2| > S-N.$ 

Suppose $z\in M\ominus (N_1 + N_2)$ and $\|z\|_2\le1$. Then we may write the Fourier decomposition of $z$ a follows:
$$z = \sum_{w\in\cC}\sum_{n\in\bZ}\alpha_{w,n}a^nwa^{-n},$$
where $\alpha_{w,n}=0$ whenever $a^nwa^{-n} \in H_1\cup H_2.$

\textbf{Claim A:} There is $\delta>0$ sufficiently small such that $\|[z,a]\|_2 < \delta$ implies that $$\sum_{w\in\cC}\sum_{n=-R}^R|\alpha_{w,n}|^2 < \ee^2.$$

\textit{Proof of Claim A:} Set $\delta = \ee^2/(48R^3)$. Following the \emph{uniform-flattening} strategy from \cite{patchellelayavalli2023sequential}, for each $n\in\bZ$ write $v_n = (\alpha_{w,n})_{w\in\cC} \in \ell^2(\cC)$. Observe that
    $$\|[z,a]\|_2^2 = \sum_{n\in\bZ}\|v_n - v_{n+1}\|^2<\delta^2.$$
    By Lemma 6.8 of \cite{patchellelayavalli2023sequential}, we have that $\|v_0\|^2 < 9\delta$. Since $\|v_n - v_{n+1}\| < \delta$ for all $n,$ we observe that $\|v_n\| \leq 3\sqrt{\delta} + |n|\delta$ for each $n\in\bZ$. Therefore $\sum_{n=-R}^R\|v_n\|^2 \leq (2R+1)(3\sqrt{\delta} + |n|\delta)^2 < 48R^3\delta = \ee^2$. Noting $\sum_{w\in\cC}\sum_{n=-R}^R|\alpha_{w,n}|^2 = \sum_{n=-R}^R\|v_n\|^2$ concludes the proof of the claim. 

\textbf{Claim B:} There is $\delta>0$ sufficiently small such that $\|[z,a]\|_2 < \delta$ implies that $$\sum_{w\in\cC}\sum_{n\in X_w}|\alpha_{w,n}|^2 \leq \ee^2.$$

\textit{Proof of Claim B:} Suppose towards a contradiction that $\sum_{w\in\cC}\sum_{n\in X_w}|\alpha_{w,n}|^2 > \ee^2$. Recall that for each $w\in\cC,$ we had that $n\in X_w$ and $k\in [2N+1,S-N]$ implies $n+k\not\in X_w,$ and moreover $n_1,n_2\in X_w$ implies that either $|n_1-n_2| \leq 2N$ or $|n_1-n_2| > S-N.$ Therefore for a fixed word $w\in\cC,$ the integers $n + k(2N+1)$ are all distinct, where $n\in X_w$ and $0\leq k < (S-N)/(2N+1).$ Since $(S-N)/(2N+1) > 1/\ee^2,$ we may simply sum over $k$ from 0 to (the floor of) $1/\ee^2.$ Thus
$$\sum_{w\in\cC}\sum_{n\in X_w}\sum_{k=0}^{1/\ee^2}|\alpha_{w,n+k(2N+1)}|^2 \leq \|z\|_2^2 \le 1. $$
For each $0\leq k < 1/\ee^2,$ define vectors $v_k = (\alpha_{w,n+k(2N+1)})_{w\in\cC,n\in X_w}$ which are each square-summable sequences such that $\sum_{k=0}^{1/\ee^2}\|v_k\|^2 \leq 1.$ Our assumption is that $\|v_0\|^2 > \ee^2$. As in the proof of Lemma 6.8 in \cite{patchellelayavalli2023sequential}, if $\|[z,a]\|_2 < \delta$ then by the Cauchy-Schwarz inequality $\|v_k\| \geq \|v_0\| - \delta\sqrt{k(2N+1)}$. Therefore, for $\delta$ sufficiently small depending only on $\ee$ and $N$, we have
\begin{align*}
    \sum_{k=0}^{1/\ee^2}\|v_k\|^2 &\geq \sum_{k=0}^{1/\ee^2}(\|v_0\| - \delta\sqrt{k(2N+1)})^2\\
    &> 1/\ee^2\cdot \ee^2 = 1,
\end{align*}
a contradiction. This proves Claim B.

Now fix a lift $x = (x_k)_k.$ Since $[x,a] = 0,$ we have that $\|[x_k,a]\|_2 \to 0.$ By the Claims, there is $T\in\cU$ so that for all $k\in T$, we can approximate $x_k$ by $x_k'$ such that $\|x_k-x_k'\|_2 \le 2\ee$ and 
$$x_k' = \sum_{w\in\cC}\sum_{n\in\bZ}\beta_{w,n}a^nwa^{-n},$$
where $\beta_{w,n} = 0$ whenever $a^nwa^{-n} \in H_1\cup H_2,$ whenever $n\in[-R,R]$, and whenever $n\in X_w.$ Therefore $x_k'$ is supported on words whose first block has some coordinate outside of $I$; $x_k'$ is also only supported on words with at least two blocks. Hence all of the words in the support of $y_1x_k'$ have at least three blocks and their first blocks have coordinates only in $[-N,N]$ (the last block of $y_1$ and the first block of each word in the support of $x_k'$ cannot cancel fully). Furthermore, consider the second block of $y_1x_k'.$ If $y_1$ has at least 3 blocks,  then the second block of $y_1x_k'$ is just the second block of $y_1$ and so it has coordinates in $[-N,N]\subset [-S,S].$ However, it is also possible for $y_1$ to have exactly 2 blocks. In this case, note that a word $g$ is in the support of $x_k'$ only if $c(g) \cap [-S,S] = \emptyset$ or there is some $n \in c(g)\cap [-S,S]\setminus [-N,N]$. In the first case, the second block of $y_1g$ retains its coordinates from $y_1,$ so it has some coordinate in $[-N,N]\subset[-S,S]$. In the second case, the second block of $y_1g$ retains its coordinate that lies in $[-S,S]\setminus [-N,N]$, so it has some coordinate in $[-S,S]$.

On the other hand, consider $x_k'y_2.$ Write $x_k' = \sum_i \gamma_i w_i a^{p_i}$ where each $\gamma_i$ is a scalar and each $w_i \in \ip{b_n,c_n : n\in\bZ}$. Consider only those words $w_i$ such that $\gamma_i$ is nonzero and $w_ia^{p_i}y_2$ is a word whose first block has only coordinates in $[-N,N]$. Since the first block of each $w_i$ has some coordinate outside of $[-R,R]$ and the blocks of $y_2$ only have coordinates inside $[-N,N]$, we must have $p_i > R-N$ for all $i.$ Since $R = S+2N,$ we have that $p_i > S+N$ for all $i.$ For these $w_i,$ we note that all blocks except the first get completely cancelled by blocks of $a^{p_i}y_2.$ Therefore the second block of $w_ia^{p_i}y_2,$ if it exists, is actually a block of $a^{p_i}y_2.$ But since every block of $y_2$ has coordinates in $[-N,N]$, every block of $a^{p_i}$ has coordinates outside $[-S,S]$. In summary, every word in the support of $x_k'y_2$ either (i) has a first block with some coordinate outside $[-N,N]$, (ii) has 1 or fewer blocks, or (iii) has a second block whose coordinates are entirely outside of $[-S,S]$. 

This means that $y_1x_k' \perp x_k'y_2$. Writing $x' = (x_k')_k\in L^2(M)^\cU,$ we have that $y_1x' \perp x'y_2$ and $\|x'-x\|_2 < 2\ee$. Since $\ee>0$ was arbitrary, we get that $y_1x\perp xy_2.$ It is then clear by taking sums and 2-norm limits (and using the fact that $x$ is operator norm bounded) that $z_1x\perp xz_2$ for all $z_1,z_2\in M\ominus (N_1+N_2)$.
    
\end{proof}

\section{Proof of main result}




The following lemma is probably well known to experts, but we include it here for completeness.
\begin{lem}\label{lem: amplification}
    Let $(M,\tau)$ be a tracial von Neumann algebra and $A,(B_n)_{n\ge1}$ be von Neumann subalgebras of $M$. Let $\Tilde{M} = \oplus_n M_n$ where $M_n\simeq M$ for all $n$, and $\Tilde{B} = \oplus_n B_n.$ Let $\Tilde{A}\subset\Tilde{M}$ be defined by the constant sequence $\Tilde{A} = \{(a)_n: a\in A\}$. We clearly have natural inclusions $\Tilde{A},\Tilde{B}\subset\Tilde{M}$. Let $(t_n)_n$ be a sequence of numbers in $(0,1]$ such that $\sum_n t_n=1.$ Give $\Tilde{M}$ a trace via $\Tilde{\tau}((x_n)_n) = \sum_n\tau(x_n)$. The following are equivalent:
    \begin{enumerate}
        \item $\Tilde{A} \prec_{\Tilde{M}} \Tilde{B}$;
        \item There exists $N\in\bN$ such that $A\prec_M B_N$.
    \end{enumerate}
\end{lem}

\begin{proof}
    (2) $\implies$ (1): Let $(\Tilde{u}_k)$ be a  sequence of unitaries in $\Tilde{A}$. Then each $\Tilde{u}_k$ is equal to a constant sequence $(u_k)_n$ for some $u_k\in Q.$ Since $A\prec_M B_N,$ there are $x,y\in M$ such that $\|E_{B_N}(xu_ky)\|_2 \not\to 0$ as $k\to\infty.$ Define $\Tilde{x} = (x)_n$ and $\Tilde{y} = (y)_n$ in $\Tilde{M}$. Then \begin{align*}
        \|E_{\Tilde{B}}(\Tilde{x}\Tilde{u_k}\Tilde{y})\|_2^2 &= \sum_n t_n\|E_{B_n}(xu_ky)\|_2^2 \\
        &\ge t_N\|E_{B_N}(xu_ky)\|_2^2 \not\to 0
    \end{align*}
    as $k\to\infty$, demonstrating that $\Tilde{A}\prec_{\Tilde{M}}\Tilde{B}$.

    (1) $\implies$ (2): By hypothesis, there is a non-zero $\Tilde{A}$-$\Tilde{B}$ subbimodule $\cH$ of $L^2(\Tilde{M})$ such that $\dim(\cH_{\Tilde{B}})<\infty$. Take $\xi\in\cH$ to be any nonzero vector; write $\xi = (\xi_n)_n$. Then for some $N$, we must have $\xi_N \neq 0.$ Since the unit of $M_N$ is in $\Tilde{B}$, we must have that the vector $\hat{\xi} = (0,\ldots,0,\xi_N,0,0,\ldots) \in \cH.$ But note that $\overline{\Span}\Tilde{A}\hat{\xi}\Tilde{B} = (0,\ldots,0,\overline{\Span}A\xi_N B_N,0,0,\ldots)$. This space is contained in $\cH$ and so has finite right $\Tilde{B}$ dimension. Clearly this forces $\overline{\Span}A\xi_N B_N$ to have finite right $B_N$ dimension, showing that $A\prec_M B_N.$
\end{proof}

\begin{lem}\label{lem: intertwining}
    Let $P,(N_k)_{k\ge1}$ be von Neumann subalgebras of a tracial von Neumann algebra $(M,\tau)$. Let $\cU$ be a free ultrafilter on $\mathbb N$ and assume that $P$ is amenable. If $P\not\prec_M N_i$ for all $k$ then there exists a unitary in $P'\cap P^\cU$ such that $u\in M^\cU \ominus (\sum_k N_k^\cU)$. 
\end{lem}

\begin{proof}
    Since $P$ is amenable, there are finite dimensional subalgebras $P_j\subset P$ such that $P_j'\cap P$ is finite index in $P$ for all $j$ and $\overline{\cup_j P_j} = P.$ Therefore $P_j'\cap P\not\prec_M N_k$ for all $j,k.$ For now, fix $j$ and set $Q = P_j'\cap P.$ Let $\Tilde{M} = \oplus_k M_k$ where $M_k\simeq M$ for all $k$, and $\Tilde{N} = \oplus_k N_k.$ Let $\Tilde{Q}\subset\Tilde{M}$ be defined by the constant sequence $\Tilde{Q} = \{(x)_n: x\in Q\}$. Let $(t_k)_k$ be a sequence of numbers in $(0,1]$ such that $\sum_k t_k=1.$ Give $\Tilde{M}$ a trace via $\Tilde{\tau}((x_k)_k) = \sum_k\tau(x_k)$. By Lemma \ref{lem: amplification}, we have that $\Tilde{Q}\not\prec_{\Tilde{M}} \Tilde{N}$. Therefore there is a unitary $u_j\in\Tilde{Q}$ such that $\|E_{\Tilde{N}}(u_j)\|_2^2 < 1/j.$ This implies that $\|E_{N_k}(u_j)\|_2^2 < 1/(jt_k)$. Now consider the unitary $u = (u_j)_j \in P^\cU$. Since $u_j\in P_j'\cap P,$ we have that $u\in P'\cap P^\cU.$ Furthermore, we have that, for all $k,$ $\|E_{N_k^\cU}(u)\|_2 = \lim_{j\to\cU}\|E_{N_k}(u_j)\|_2 = 0.$ This proves the lemma.
\end{proof}

\begin{thm}\label{thm: generalize cyril}
    Let $(M,\tau)$ be a tracial von Neumann algebra. Let $A\subset N_1,\ldots,N_k \subset M$ be amenable von Neumann subalgebras. Assume the following:
    \begin{enumerate}
        \item $N_i$ is weakly mixing through $A$ in $M$ for all $1\leq i\leq k$;
        \item $\{N_1,\ldots,N_k\}$ has the asymptotic orthogonality property in $M$ relative to $A$.
    \end{enumerate}
    Then for any amenable extension $P\supset A$, there exist projections $p_1,\ldots,p_k \in P'\cap N_1\cap\cdots\cap N_k$ such that $\sum_i p_i = 1$ and $Pp_i \subset p_iN_ip_i$ for all $1\leq i\leq k$.
\end{thm}

\begin{proof}
    We proceed as in the proof of Theorem 8.1 in \cite{HOUDAYER2014414}. For clarity, we assume without loss of generality that $k=2$. Let $P$ be an amenable extension of $A.$ Using weak mixing, $P'\cap M \subset N_i$ for $i=1,2$. Note that in particular this implies that for $P'\cap M \subset N_1\cap N_2$; we denote $N = N_1\cap N_2.$ Let $p_1$ be the maximal projection in $P'\cap N$ such that $Pp_1 \subset p_1N_1p_1.$ Let $p_2$ be the maximal projection in $(1-p_1)(P'\cap N)(1-p_1)$ such that $Pp_2 \subset p_2N_2p_2.$ Assume towards a contradiction that $p_1+p_2\neq1.$ Write $q = 1-p_1-p_2$ and set $Q = Pq$. Assume towards a contradiction that $Q\prec_M N_i$ for $i=1,2$. By Lemma \ref{lem: intertwining}, there is a unitary $u = (u_n)\in Q'\cap Q^\cU$ such that $u\in P' \cap M^\cU\ominus (N_1^\cU+N_2^\cU)$. Using simultaneous relative asymptotic orthogonality, we get that for $n$ sufficiently large, $uu_n$ is both nearly equal to and nearly orthogonal to $u_nu$, a contradiction. We therefore conclude that for some $i=1,2$, we have $Q\prec_M N_i$. Without loss of generality we assume $Q\prec_M N_2,$ with the case $i=1$ being similar. By Theorem \ref{thm-popa-fundamental}, there is $n\geq 1$, a nonzero partial isometry $v = [v_1 \cdots v_n] \in M_{1,n}(qM)$, a projection $r\in M_n(N_2)$, and a *-homomorphism $\theta:Q \to rM_n(N_2)r$ such that $v^*v \leq r$ and $xv = v\theta(x)$ for all $x\in Q.$ In particular, we see that for each $i,$ $Qv_i \subset \sum_j v_jN_2$, whence $Av_i \subset \sum_jv_jN_2$. By Lemma \ref{lem: wk-mixing-absorb}, since $N_2\subset M$ is weakly mixing through $A$, we have that $v_i\in N$ for each $i.$ Therefore we deduce that $vv^* \in Q'\cap qN_2q$ and $Qvv^*\subset vv^*N_2vv^*$. Hence $P(p_2 + vv^*)\subset (p_2+vv^*)N_2(p_2+vv^*)$. But this contradicts the maximality of $p_2$, concluding the proof.
\end{proof}

We can now prove the results stated in the introduction.

\begin{proof}[Proof of Theorem \ref{beast1}]
    By Lemma \ref{lem: amalgam-mixing}, we have that $N_i$ is weakly mixing through $A$ in $M$ for all $1\leq i\leq k$. By Theorem \ref{thm: asymptotic orthogonality} we have that $\{N_1,\ldots,N_k\}$ has the asymptotic orthogonality property in $M$ relative to $A$. Thus Theorem \ref{thm: generalize cyril} says that any amenable extension $P\supset A$ contains projections $p_1,\ldots,p_k \in P'\cap N_1\cap\cdots \cap N_k$ such that $\sum_i p_i = 1$ and $Pp_i \subset p_iN_ip_i$ for all $1\leq i\leq k$. Since $A\subset M$ is mixing, we moreover have that each $p_i$ commutes with $A$ and is therefore in $A\subset P.$ Each $p_i$ is therefore central in $P$, whence every factorial amenable extension of $A$ is contained in some $N_i.$ Evidently each $N_i$ is a distinct factorial maximal amenable extension of $A$ in $M$.
\end{proof}

\begin{proof}[Proof of Corollary \ref{coro of beast}]
    By Theorem \ref{beast1}, every amenable extension of $A$ is contained in an algebra of the form $\bigoplus_{i=1}^k p_iN_ip_i$ where $p_i\in A$ for all $i$ and $\sum_i p_i = 1.$ It is thus clear that the maximal amenable extensions are exactly of the form $\bigoplus_{i=1}^k p_iN_ip_i$ where $p_i\in A$ for all $i$ and $\sum_i p_i = 1.$ We show that two algebras $\bigoplus_{i=1}^k p_iN_ip_i$ and $\bigoplus_{i=1}^k q_iN_iq_i$ are unitarily conjugate if and only if $\tau(p_i) = \tau(q_i)$ for each $i.$ Once we know this, we see that every maximal amenable extension of $A$ is, up to unitary conjugacy, uniquely determined by the values of $\tau(p_i)$. Since these are non-negative numbers summing to 1, we can view each maximal amenable extension of $A$ as a unique convex combination of the factors $N_i$, giving us the $k$ dimensional $\bR$-simplex.

    Write $P = \bigoplus_{i=1}^k p_iN_ip_i$ and $Q = \bigoplus_{i=1}^k q_iN_iq_i$. If $\tau(p_i) = \tau(q_i)$ for each $i,$ then since each $N_i$ is a II$_1$ factor, there are partial isometries $v_i \in N_i$ such that $v_iv_i^* = p_i $ and $v_i^*v_i = q_i$. Set $u = \sum_i v_i$. It is routine to verify that $u$ is a unitary in $M$ and $u^*Pu = Q.$ Conversely, suppose that there is some unitary $u\in M$ such that $u^*Pu = Q$. Then we must have $u^*Z(P)u = Z(Q)$ and thus there is a permutation $\sigma$ such that for all $i$, $u^*p_iu = q_{\sigma(i)}$. This implies that $N_i \prec_M N_{\sigma(i)}$. We show that $\sigma$ must be the identity permutation. Assume towards a contradiction that for some $i,$ $\sigma(i)\neq i$. Recall that we may view $M = L(G)$ where $G = *_K H_\ell$ where $1\leq \ell \leq k$, $H_\ell = \bZ\wr\bZ$, and $K$ is the acting group in each $H_\ell$. Let $g$ be some element in the normal subgroup $\bigoplus_{n\in\bZ}\bZ$ of $H_i.$ By the intertwining theorem, if $N_i\prec_M N_{\sigma(i)}$ then for some $x,y\in M$, $\|E_{N_{\sigma(i)}}(xu_g^ny)\|_2 \not\to 0.$ To contradict this, we show that $\|E_{N_{\sigma(i)}}(xu_g^ny)\|_2 \to 0$ whenever $x,y$ are of the form $u_h,u_k$ for $h,k\in G$ (then we may take sums and limits). But $E_{N_{\sigma(i)}}(u_hu_g^nu_k)\neq 0$ exactly when $hg^nk \in H_{\sigma(i)}$. A routine analysis of words in amalgamated free products shows that this is only possible for finitely many values of $n$ (eventually the powers of $g$ are so large that regardless of the values of $h$ and $k,$ the reduced form of $hg^nk$ has a factor in $H_i\setminus K$). So it cannot be that $N_i\prec_M N_{\sigma(i)}$; thus, we conclude that $\sigma$ is the identity permutation. Therefore we have $u^*p_iu = q_i$ for all $i$ and in particular $\tau(p_i) = \tau(q_i)$ for all $i.$
\end{proof}

\bibliographystyle{plain}
\bibliography{probbib}

\end{document}